\documentclass[12pt]{amsart}
\usepackage{amsmath, amssymb, amsthm, microtype}
\usepackage[bottom=1in, left=1.15in, right=1.15in, top=1in]{geometry}
\usepackage{color}

\makeatletter
\renewcommand{\pod}[1]{\mathchoice
  {\allowbreak \if@display \mkern 18mu\else \mkern 8mu\fi (#1)}
  {\allowbreak \if@display \mkern 18mu\else \mkern 8mu\fi (#1)}
  {\mkern4mu(#1)}
  {\mkern4mu(#1)}
}

\newtheorem{theorem}{Theorem}[section]

\newtheorem{lemma}[theorem]{Lemma}

\newtheorem{proposition}[theorem]{Proposition}

\theoremstyle{remark}

\newtheorem{example}[theorem]{Example}

\theoremstyle{definition}

\newcommand{\ord}{\mathrm{ord}}

\newcommand{\cC}{\mathcal{C}}

\newcommand{\cT}{\mathcal{T}}

\newcommand{\eps}{\epsilon}

\newcommand{\bZ}{\mathbb{Z}}

\newcommand{\R}{\mathbb{R}}
\newcommand{\Z}{\mathbb{Z}}
\newcommand{\Q}{\mathbb{Q}}

\newcommand{\mfg}{\mathfrak{g}}
\newcommand{\mfa}{\mathfrak{a}}

\newcommand{\mfj}{\mathfrak{j}}
\newcommand{\mfm}{\mathfrak{m}}

\newcommand{\mfp}{\mathfrak{p}}
\newcommand{\mfq}{\mathfrak{q}}

\newcommand{\Id}{\mathrm{Id}}
\newcommand{\PrinFrac}{\mathrm{PrinFrac}}
\newcommand{\Cl}{\mathrm{Cl}}

\AtBeginDocument{%
   \def\MR#1{}
}

\title[Irreducible elements in rings of integers of number fields]{Two problems concerning irreducible elements in rings of integers of number fields}

\begin{document}

\author{Paul Pollack}
\address{Department of Mathematics, Boyd Graduate Studies Research Center, University of Georgia, Athens, GA 30602, United States}
\email{pollack@uga.edu}

\author{Lee Troupe}
\address{Department of Mathematics, University of British Columbia, Vancouver, British Columbia V6T 1Z2, Canada}
\email{ltroupe@math.ubc.ca}

\begin{abstract} Let $K$ be a number field with ring of integers $\Z_K$. We prove two asymptotic formulas connected with the distribution of irreducible elements in $\Z_K$. First, we estimate the maximum number of nonassociated irreducibles dividing a nonzero element of $\Z_K$ of norm not exceeding $x$ (in absolute value), as $x\to\infty$. Second, we count the number of irreducible elements of $\Z_K$ of norm not exceeding $x$ lying in a given arithmetic progression (again, as $x\to\infty$). When $K=\Q$, both results are classical; a new feature in the general case is the influence of combinatorial properties of the class group of $K$.
\end{abstract}

\maketitle

\section{Introduction}  Let $K$ be an algebraic number field with corresponding ring of integers $\Z_K$. In this paper, we take two well-known analytic results about rational prime numbers and prove  generalizations for the irreducible elements of $\Z_K$. Our results can be seen as contributing to the program of understanding how combinatorial attributes of the class group of $K$ influence factorization properties of the domain $\Z_K$. Several previous results in this direction are discussed in Chapter 9 of Narkiewicz's monograph \cite{narkiewicz04}.

For each nonzero $\alpha \in \Z_K$, we let $\nu(\alpha)$ denote the number of nonassociate irreducible elements of $\Z_K$ dividing $\alpha$. We view $\nu(\cdot)$ as a generalization to $\Z_K$ of the classical arithmetic function $\omega(\cdot)$, which counts the number of distinct prime factors of its (rational integer) argument.

A 1940 theorem of Erd\H{o}s and Kac asserts, roughly speaking, that $\omega(n)$ is normally distributed with mean $\log \log |n|$ and standard deviation $\sqrt{\log \log |n|}$.  More precisely, for each fixed real $u$,
\[ \frac{\#\{3 \le n \le x: \omega(n) \le \log \log n + u \sqrt{\log\log{n}}\}}{\#\{3\le n \le x\}} \to \frac{1}{\sqrt{2\pi}} \int_{-\infty}^{u} e^{-t^2/2}\, dt, \]
as $x\to\infty$. This theorem of Erd\H{o}s and Kac sharpens, in a striking way, the 1917 result of Hardy and Ramanujan \cite{hr17} that $\omega(n)$ has normal order $\log \log n$.  Recently \cite{poleek}, the first-named author proved the following generalization of Erd\H{o}s--Kac: For every number field $K$, there are positive constants $A$ and $B$, and a positive integer $D$, such that $\nu(\alpha)$ is normally distributed with mean $A(\log\log{|N(\alpha)|})^{D}$ and standard deviation $B(\log \log {|N(\alpha)|})^{D-1/2}$. The constants $A, B$, and $D$ depend on $K$ only through its class group. When the class group is trivial, $A=B=D=1$, and so this result does indeed contain the original Erd\H{o}s--Kac theorem.

While $\omega(n)$ is usually quite close to $\log\log{n}$, it occasionally gets much larger. For each $x \ge 2$, the maximum value of $\omega(n)$ on the integers $n\le x$ is assumed when $n$ is the product of a certain initial segment of primes. This easy observation, together with the prime number theorem, implies in a straightforward way (cf. \cite[p. 471]{hw00}) that
\begin{equation}\label{eq:omegalimsup} \max_{n\le x} \omega(n) = (1+o(1))\frac{\log{x}}{\log\log{x}}, \end{equation}
as $x\to\infty$. Our first goal is to supply a corresponding description of the maximum value of $\nu(\alpha)$ for $\alpha \in \Z_K$ with $|N(\alpha)| \le x$, for an arbitrary number field $K$. The answer in the general case is decidedly more interesting than the case $K=\Q$, and the proof, while similar in spirit, is correspondingly more subtle.

For a finite abelian group $G$ (described multiplicatively), we let $D(G)$ denote the \emph{Davenport constant of $G$}, i.e., the smallest positive integer with the property that any sequence of elements of $G$ of length $D(G)$ possesses a nonempty subsequence whose product is the identity. We write $D$ for the Davenport constant of the class group of $K$. (This is the same value of $D$ that appears in the result from \cite{poleek} quoted above.) We let $h$ denote the class number of $K$. Our first main theorem is the following.

\begin{theorem}\label{thm:1stmain} As $x\to\infty$, we have
\[
\max_{\alpha:~0<|N(\alpha)|\le x} \nu(\alpha) = (1+o(1)) \cdot M \cdot \Big( \frac{\log x}{h\log\log x} \Big)^D.
\]
Here $M$ is a positive constant, depending only on the class group of $K$, to be specified in the course of the proof.
\end{theorem}

We turn now to our second theme, the distribution of irreducibles in arithmetic progressions. Let $\Pi(x)$ denote the count of nonassociate irreducibles $\pi \in \Z_K$ with $|N(\pi)|\le x$. So when $K=\Q$, we have $\Pi(x) = \pi(x)$, the familiar rational prime counting function. It was shown by R\'emond in 1966 \cite[Chapter 2]{remond66} that for any number field $K$,
\[ \Pi(x) = (C+o(1)) \frac{x}{\log{x}} (\log\log{x})^{D-1} \]
for a constant $C >0$ depending only on the class group of $K$. (We continue to use $D$ for the Davenport constant of the class group.) When the class group is trivial, $C=D=1$, and so this result contains the classical prime number theorem.

Our second main result is a corresponding generalization of the prime number theorem for arithmetic progressions. For $\mfm$ a nonzero ideal of $\Z_K$ and $\alpha$ a nonzero element of $\Z_K$, we let  $\Pi(x;\mfm,\alpha)$ denote the following counting function of irreducibles:
\[ \Pi(x;\mfm,\alpha) =\#\{\text{principal ideals }(\pi): |N\pi| \le x, \pi\text{ irred.}, \pi \equiv \alpha\pmod{\mfm},\text{ and } \frac{\pi}{\alpha} \gg 0\}. \]
Here the notation ``$\gg 0$''  indicates total positivity. (There will be no danger of confusion with Vinogradov's notation for orders of magnitude.) It is clear that if $(\alpha)$ and $\mfm$ have a common principal ideal divisor in $\Z_K$ other than the unit ideal, then $\Pi(x;\mfm,\alpha)\le 1$ for all $x$. If there is no such principal ideal, we will say that $\alpha$ and $\mfm$ are \emph{weakly relatively prime}. (We reserve the term \emph{relatively prime} for the case when $\alpha$ and $\mfm$ have no nonunit ideal divisors at all in common, or equivalently, for when $\alpha$ and $\mfm$ are comaximal.)

\begin{theorem}\label{thm:2ndmain} Let $\mfm$ be a nonzero ideal of $\Z_K$ and let $\alpha$ be a nonzero element of $\Z_K$ with $\alpha$ and $\mfm$ weakly relatively prime. Then there is a positive constant $C'$ and a positive integer $L$ such that, as $x\to\infty$,
\[ \Pi(x;\mfm,\alpha) = (C'+o(1)) \frac{x}{\log{x}} (\log\log{x})^{L-1}. \]
The constants $C'$ and $L$ will be specified in the course of the proof.
\end{theorem}

It is obvious by comparison with R\'emond's theorem that $L \le D$, for all choices of $\alpha$ and $\mfm$. It will emerge from the proof that (for $K$ fixed) the constants $C'$ and $L$ depend only on the gcd ideal $(\alpha,\mfm)$, with $L=D$ precisely when $(\alpha,\mfm)$ is the unit ideal. Consequently, 100\% of irreducibles are (strongly) relatively prime to $\mfm$, and those are asymptotically uniformly distributed among the strict ray classes modulo $\mfm$ represented by principal ideals prime to $\mfm$.

\section{Background on the equidistribution of prime ideals in ray classes}
In both of our main results, the key analytic input is a 1918 theorem of Landau \cite{lan18} on the equidistribution of prime ideals in strict ray classes. We recall the setup here. For detailed proofs of the basic facts used about ray class groups, see the recent book of Childress \cite[Chapter 3]{childress09}.

For a number field $K$, we let $\Id(K)$ and $\PrinFrac(K)$ denote the group of all fractional ideals of $K$ and all principal fractional ideals of $K$, respectively. For each nonzero integral ideal $\mfm$ of $K$, let
\[ \Id_{\mfm}(K) = \{\text{fractional ideals }\mfa: \mathrm{ord}_\mfp(\mfa) = 0\text{ for all prime ideals $\mfp \mid \mfm$}\},\]
and let
\[ \PrinFrac_{\mfm}^{+}(K) = \{\gamma\Z_K: \gamma \in K, \gamma \equiv 1~\mathrm{mod}^{+}~{\mfm}\}; \]
here the $\mathrm{mod}^{+}$ notation means that $\mathrm{ord}_\mfp(\gamma - 1) \ge \mathrm{ord}_\mfp(\mfm)$ for all prime ideals $\mfp \mid \mfm$ and that $\gamma \gg 0$. The strict ray class group of $K$ modulo $\mfm$ is defined by
\[ \Cl_{\mfm}(K) := \Id_{\mfm}(K)/\PrinFrac_{\mfm}^{+}(K). \] The order of $\Cl_{\mfm}(K)$, referred to as the \emph{strict ray class number mod $\mfm$}, is denoted $h_{K,\mfm}$. We view the partition of ideals prime to $\mfm$ into strict ray classes as analogous to the partition of rational integers coprime to $m$ into residue classes mod $m$. (In fact, if $K=\Q$ and $\mfm=(m)$, then $\Cl_{\mfm}(K) \cong (\Z/m\Z)^{\times}$, via the map $[ab^{-1}\Z_K] \mapsto ab^{-1}\bmod{m}$.)

For each strict ray class $\cC \in \Cl_{\mfm}(K)$, we define the prime ideal counting function
\[ \pi_K(x;\cC) := \sum_{\substack{\mfp:~N(\mfp) \le x \\ \mfp \in \cC}} 1. \]

\begin{theorem}[Landau's equidistribution theorem]\label{thm:landau} Fix a $\cC \in \Cl_{\mfm}(K)$. For all $x\ge 3$,
\[ \pi_K(x;\cC) = \frac{1}{h_{K,\mfm}} \int_{2}^{x} \frac{dt}{\log{t}} + O_K(x \exp(-c_K \sqrt{\log x})). \]
Here $c_K$ is a positive constant depending only on $K$.
\end{theorem}

The following remarks will be useful in our applications. Recall that the (ordinary) class group of $K$, here denoted $\Cl(K)$, is defined as $\Id(K)/\PrinFrac(K)$. The inclusion $\Id_{\mfm}(K) \hookrightarrow \Id(K)$ induces an isomorphism
\[ \Cl(K) = \Id(K)/\PrinFrac(K) \cong \Id_{\mfm}(K)/\PrinFrac_{\mfm}(K), \]
where $\PrinFrac_{\mfm}(K) = \Id_{\mfm}(K) \cap \PrinFrac(K)$. In particular, letting $h_K := \#\Cl(K)$ (the ordinary class number),
\begin{align*} h_K = [\Id_{\mfm}(K): \PrinFrac_{\mfm}(K)] &= \frac{[\Id_{\mfm}(K): \PrinFrac_{\mfm}^{+}(K)]}{[\PrinFrac_{\mfm}(K): \PrinFrac_{\mfm}^{+}(K)]} \\ &= \frac{h_{K,\mfm}}{[\PrinFrac_{\mfm}(K): \PrinFrac_{\mfm}^{+}(K)]}. \end{align*}
Rearranging,
\begin{equation}\label{eq:isunion} [\PrinFrac_{\mfm}(K): \PrinFrac_{\mfm}^{+}(K)] = \frac{h_{K,\mfm}}{h_{K}}. \end{equation}
Thus, the ratio $h_{K,\mfm}/h_{K}$ can be interpreted as the number of strict ray classes modulo $\mfm$ represented by principal ideals prime to $\mfm$. Motivated by the analogy with Euler's totient function, we set
\[ \Phi(\mfm) := \frac{h_{K,\mfm}}{h_K}. \]

Various earlier algebro-analytic results can be recovered as special cases of Theorem \ref{thm:landau}. For instance, we easily deduce the equidistribution of prime ideals relative to the ordinary class group $\Cl(K)$; in that case, the factor $\frac{1}{h_{K,\mfm}}$ in Theorem \ref{thm:landau} should be  replaced with $\frac{1}{h_K}$. (To see this implication, take $\mfm=(1)$ in Theorem \ref{thm:landau} and note that, by \eqref{eq:isunion}, each  ideal class modulo $(1)$ is a union of $h_{K,(1)}/h_K$ strict ray classes mod $(1)$.) Theorem \ref{thm:landau} also implies the \emph{prime ideal theorem}, that the total number of prime ideals of norm not exceeding $x$ is asymptotically $\int_{2}^{x} dt/\log{t}$. (Sum over all $h_{K,\mfm}$ strict ray classes.) Since these earlier theorems are also due to Landau, we shall refer to any and all of these results as ``Landau's theorem''.


\section{Anatomy of an irreducible}\label{sec:anatomy}
Since ideals of $\Z_K$ factor uniquely while elements typically do not, we will rephrase our questions about irreducible elements in ideal-theoretic terms. To make this translation, it is important to understand how irreducible elements decompose as products of prime ideals. We recall the basic results here (cf. \cite[\S9.1]{narkiewicz04}).

Fix --- once and for all --- an ordering of the (ordinary) ideal classes, say $\cC_1, \dots, \cC_h$, where $h=h_K$.

We define a \emph{type} as an $h$-tuple of nonnegative integers. Given a nonzero integral ideal $\mfa$ of $\Z_K$, the \emph{type of $\mfa$} is the tuple $(t_1,\dots,t_h)$, where $t_i$ is the number of prime ideals dividing $\mfa$ from the class $\cC_i$, counted with multiplicity. The type of a nonzero element of $\Z_K$ is defined as the type of the corresponding principal ideal.

Now let $\pi$ be an irreducible element of $\Z_K$. Irreducibility implies that the prime ideal factorization of $(\pi)$ has no nonempty proper subproduct equal to a principal ideal. Thus, if $\tau=(t_1,\dots,t_h)$ is the type of an irreducible, then $\tau$ has the property that $\cC_1^{t_1} \cdots \cC_h^{t_h}$ is trivial in $\Cl(K)$ but no proper nonempty subproduct is trivial. If $\tau=(t_1,\dots,t_h)$ is any type with this property and not all $t_i=0$, we call $\tau$ an \emph{irreducible type}.

Landau's theorem implies that every irreducible type $(t_1,\dots,t_h)$ is the type of an irreducible element. Indeed, if we multiply $t_i$ prime ideals from the class $\cC_i$ (for $i=1,2,\dots,h$), the result is a principal ideal, and each of its generators is an irreducible of the sought-after kind. (Landau's theorem is used here only to ensure the existence of at least one prime ideal in each ideal class.) Thus, the irreducible types are exactly the types of irreducibles.

Recall our notation $D$ for the Davenport constant of $\Cl(K)$.
It follows quickly from the definition of the Davenport constant that if $(t_1,\dots,t_h)$ is any irreducible type, then $t_1+\dots+t_h\le D$ (so that, in particular, there are only finitely many irreducible types) and that equality holds for some irreducible type. The quantity $t_1+\dots+t_h$ will be referred to as the \emph{length} of $\tau=(t_1,\dots,t_h)$. We call an irreducible type with length $D$ a \emph{maximal irreducible type}.

\section{The maximal order of $\nu$: Proof of Theorem \ref{thm:1stmain}}
Let $\cT_{\max}$ denote the collection of maximal irreducible types, and consider the polynomial in $x_1,\dots,x_h$ defined by
\[
P(x_1, \ldots, x_h) = \sum_{\tau \in \cT_{\max}} \prod_{i = 1}^h \frac{x_i^{t_i}}{t_i!},
\]
where we have written each $\tau$ as $(t_1,\dots,t_h)$. Note that $P$ depends on $K$ only via its class group. Let $M$ denote the maximum value of $P$ on the simplex
\[
\Delta = \{(x_1, \ldots, x_h)\in \R^h: x_i \geq 0 \,\,\, \forall i, \,\,\, \sum x_i \leq h\}.
\]
The value $M$ exists, as the maximum of a polynomial on a compact set, and it is obvious that $M > 0$. We will show that Theorem \ref{thm:1stmain} holds with this value of $M$.

The proof goes in two stages. First, we show the lower bound implicit in Theorem \ref{thm:1stmain}.


\begin{lemma}{\label{lowerbound}} As $x\to\infty$,
\begin{align*}
\max_{\alpha:~0 < |N(\alpha)| \le x} \nu(\alpha) \geq (M-o(1)) \Big( \frac{\log x}{h \log\log x} \Big)^D.
\end{align*}
\end{lemma}

For the upper bound, it is convenient to extend the definition of $\nu$ as follows. For each nonzero ideal $\mfa$ of $\Z_K$, we let $\nu(\mfa)$ denote the number of nonassociate irreducibles $\pi$ for which $(\pi)$ divides $\mfa$. Thus, if $\mfa = (\alpha)$, then $\nu(\mfa) = \nu(\alpha)$.

\begin{lemma}\label{boundmatch} For each nonzero integral ideal $\mfa$ of $\Z_K$ with $N(\mfa) \le x$,
\[
\nu(\mfa) \le (M + o(1)) \Big( \frac{\log x}{h \log\log x} \Big)^D.
\]
\end{lemma}
Restricting to principal ideals in Lemma \ref{boundmatch} yields the upper bound half of Theorem \ref{thm:1stmain}. Thus, it remains only to prove Lemmas \ref{lowerbound} and \ref{boundmatch}.

\subsection{Proof of Lemma \ref{lowerbound}.} It is enough to show that for each fixed $\epsilon> 0$ and all $x> x_0(\epsilon)$, there is a nonzero $\alpha\in\Z_K$ with $|N(\alpha)|\le x$ and
\begin{equation}\label{eq:epsversion} \nu(\alpha) \ge (M -\epsilon) \left(\frac{\log{x}}{h\log\log{x}}\right)^{D}.\end{equation}
%
%
Let $\delta > 0$ be a parameter to be chosen later in terms of $\eps$, and put $X = x^{1 - \delta}$. We fix a point $(\gamma_1,\dots,\gamma_h) \in \Delta$ at which $P$ achieves its maximum. Let $\mfa$ be the (integral) ideal of $\Z_K$ defined by
\[
\mfa := \prod_{i = 1}^h \prod_{\substack{\mfp \in \cC_i \\ N(\mfp) \leq \gamma_i \log X}} \mfp.
\]
Landau's equidistribution theorem implies that each term of the inside product is of size $X^{\gamma_i/h+o(1)}$ (as $x\to\infty$); since $\sum \gamma_i\le h$,
\[
N(\mfa) \leq X^{1 + o(1)} = x^{1-\delta+o(1)}.
\]
(To estimate the product we used a form of Landau's result where prime ideals are counted with weight $\log N(\mfp)$ rather than weight $1$; this may be deduced by partial summation from Theorem \ref{thm:landau} by a standard calculation.)
This upper bound implies that, for $x$ sufficiently large, $\mfa$ has a principal multiple with norm at most $x$; indeed, it suffices to multiply $\mfa$ by the smallest ideal in $[\mfa]^{-1}$. So it is enough to show that $\nu(\mfa)$ is at least as large as the right-side of \eqref{eq:epsversion}.

We establish this bound by counting, for each maximal type $\tau$, the number of nonassociated irreducibles $\pi$ of type $\tau$ with $(\pi) \mid \mfa$. Since $\mfa$ is a product of distinct prime ideals, the number of these for a given $\tau$ is exactly
\begin{equation}\label{eq:prodformula}
\prod_{i = 1}^h \binom{\omega_i(\mfa)}{t_i},
\end{equation}
where $\omega_i(\mfa)$ is the number of distinct prime ideal divisors of $\mfa$ from the class $\cC_i$. By Landau's theorem, for each $i$ with $\gamma_i \ne 0$,
\begin{align}\notag \omega_i(\mfa) &= \sum_{\substack{\mfp \in \cC_i \\ N(\mfp) \leq \gamma_i \log X}} 1\\
 &= (\gamma_i+o(1)) \frac{\log{X}}{h\log\log{X}}.\label{eq:moreprecise}\end{align}

The terms of the product \eqref{eq:prodformula} with $t_i = 0$ are identically 1; for the others,
\begin{align*}
\binom{\omega_i(\mfa)}{t_i} &= \frac{\omega_i(\mfa) (\omega_i(\mfa)-1) \cdots (\omega_i(\mfa)-(t_i-1))}{t_i!} \\ &= \frac{\omega_i(\mfa)^{t_i}}{t_i!} + O\Big( \Big( \frac{\log X}{\log\log X} \Big)^{t_i - 1} \Big).
\end{align*}
As a consequence, the number of nonassociate irreducible divisors of $\mfa$ of type $\tau$ is
\begin{equation}\label{eq:tosumover}
\prod_{\substack{1 \leq i \leq h \\ t_i \neq 0}} \frac{\omega_i(\mfa)^{t_i}}{t_i!} + O((\log{X}/\log\log{X})^{D-1}).
\end{equation}

Suppose first that whenever $t_i \neq 0$, we have $\gamma_i \neq 0$. In that case, substituting in the estimate \eqref{eq:moreprecise} for $\omega_i(\mfa)$ yields
\begin{align}\label{numberofirreds}
\prod_{\substack{1 \leq i \leq h \\ t_i \neq 0}} \frac{\omega_i(\mfa)^{t_i}}{t_i!} &= (1 + o(1)) \prod_{\substack{1 \leq i \leq h \\ t_i \neq 0}} \frac{\gamma_i^{t_i}}{t_i!} \Big(\frac{\log X}{h \log\log X} \Big)^{t_i} \nonumber \\
 &= \Big( \prod_{1 \leq i \leq h} \frac{\gamma_i^{t_i}}{t_i!} \Big) \Big(\frac{\log X}{h \log\log X} \Big)^{D} + o\Big( \Big( \frac{\log X}{h\log\log X} \Big)^D \Big).
\end{align}
If $\gamma_i = 0$ for some $i$ with $t_i \neq 0$, then $\prod_{\substack{1 \leq i \leq h,~t_i \neq 0}} \frac{\omega_i(\mfa)^{t_i}}{t_i!} = 0=\prod_{1 \leq i \leq h} \frac{\gamma_i^{t_i}}{t_i!}$, and hence \eqref{numberofirreds} remains valid.

Inserting \eqref{numberofirreds} back into \eqref{eq:tosumover} and then summing over maximal types $\tau$, we conclude that there are
\[
M \bigg(\frac{\log X}{h \log\log X} \bigg)^{D} + o\bigg( \Big( \frac{\log X}{h\log\log X} \Big)^D \bigg)
\]
nonassociate irreducible divisors of $\mfa$. This expression is
$\sim M (1 - \delta)^D \Big(\frac{\log x}{h \log\log x} \Big)^{D}$. Choosing $\delta$ sufficiently small in terms of $\eps$ gives a lower bound exceeding the right-hand side in \eqref{eq:epsversion}.

\subsection{Proof of Lemma \ref{boundmatch}.}
Fix an ideal $\mfa$ with $\nu(\mfa)$ maximal among all nonzero ideals with $N(\mfa) \leq x$. We may assume that $\mfa$ has the following property: Whenever a prime ideal $\mfp$ divides $\mfa$, every prime ideal $\mfp'$ belonging to the same ideal class of $\mfp$ with $N(\mfp') < N(\mfp)$ also divides $\mfa$. If not, then define a new ideal $\mfa'$ by replacing $\mfp$ by $\mfp'$ in the ideal factorization of $\mfa$; then $\mfa'$ has the same number of irreducible divisors as $\mfa$, and $N(\mfa') < N(\mfa) \leq x$. We can repeat this procedure as necessary until we obtain an ideal with the desired property.

We may restrict ourselves to counting irreducibles $\pi$ dividing $\mfa$ which are (a) of maximal type and (b) have $(\pi)$ squarefree. Indeed, if either (a) or (b) fails, then $(\pi)$ is composed of at most $D-1$ distinct prime ideals. The number of distinct prime ideals dividing $\mfa$ is $\ll \log{x}/\log\log{x}$ (by a proof entirely analogous to that of \eqref{eq:omegalimsup}). Hence, the number of possibilities for the set of prime ideals dividing $(\pi)$ is $\ll (\log{x}/\log\log{x})^{D-1}$. Furthermore, having chosen the set of prime ideals dividing $(\pi)$, the number of $(\pi)$ composed of those prime ideals is $O(1)$; one can see this by noting, for instance, that the exponent to which each prime ideal can appear is bounded (e.g., by $h$). Thus, there are $\ll (\log x/\log\log{x})^{D-1}$ irreducible divisors of $\mfa$ not satisfying (a) and (b), and this will be negligible for us.

For each $1 \leq i \leq h$, define $u_i$ so that the largest prime ideal dividing $\mfa$ and belonging to the ideal class $\cC_i$ has norm $u_i \log{x}$; if no such prime ideal exists, set $u_i = 0$. Set
\[
\mfa_0 := \prod_{\substack{1 \leq i \leq h \\ u_i \neq 0}} \prod_{\substack{\mfp \in \cC_i \\ N(\mfp) < u_i \log{x}}} \mfp.
\]
Notice that $\mfa_0 \mid \mfa$, and so $$ N(\mfa_0) \le N(\mfa) \le x.$$ If $u_i > 1/\log\log x$, then we may deduce from Landau's theorem that
\[
\prod_{\substack{\mfp \in \cC_i \\ N(\mfp) < u_i \log(x)}} N(\mfp) \geq x^{u_i(1 + o(1))/h}.
\]
Hence,
\begin{equation}\label{eq:smalluisum}
\sum_{i:~u_i > 1/\log\log x} u_i \leq h + o(1).
\end{equation}

Proceeding as in the lower bound argument, we see that the count of nonassociated irreducibles $\pi$ dividing $\mfa$ with $(\pi)$ squarefree and $\pi$ having a given maximal type $\tau=(t_1,\dots,t_h)$ is
\begin{align*}
\prod_{i = 1}^h \binom{\omega_i(\mfa)}{t_i} &= \prod_{\substack{1 \leq i \leq h \\ t_i \neq 0}} \Big( \frac{\omega_i(\mfa)^{t_i}}{t_i!} + O\Big( \Big( \frac{\log x}{\log\log x} \Big)^{t_i - 1} \Big) \Big) \\
 &= \prod_{\substack{1 \leq i \leq h \\ t_i \neq 0}} \frac{\omega_i(\mfa)^{t_i}}{t_i!} + O\Big( \Big( \frac{\log x}{\log\log x} \Big)^{D - 1} \Big).
\end{align*}
To control the error terms, we used here that for each $i$,
\begin{align}\label{trivialomegabound}
\omega_i(\mfa) = O(\log x/\log\log x),
\end{align}
which follows from having each $u_i \leq 2h$, say (which in turn follows, for large $x$, from \eqref{eq:smalluisum}). Suppose first that $u_i > 1/\log\log x$ for each index $i$ with $t_i \neq 0$. Then for each $i$ with $t_i \ne 0$, Landau's theorem shows that $\omega_i(\mfa) =(1+o(1)) \frac{u_i\log{x}}{h\log\log{x}}$, and so
\[
\prod_{\substack{1 \leq i \leq h \\ t_i \neq 0}} \frac{\omega_i(\mfa)^{t_i}}{t_i!} = \prod_{i = 1}^h \frac{u_i^{t_i}}{t_i!} \Big( \frac{\log x}{h \log\log x} \Big)^D + o\Big( \Big( \frac{\log x}{h \log\log x} \Big)^D \Big).
\]
Suppose now that some $u_i \leq 1/\log\log x$. Then for this index $i$, we have $u_i\log{x} \le \log{x}/\log\log{x}$, so that (by Landau's theorem again)
\[
\omega_i(\mfa) = O\Big( \frac{\log x}{(\log\log x)^2} \Big).
\]
It follows that
\[
\prod_{\substack{1 \leq i \leq h \\ t_i \neq 0}} \frac{\omega_i(\mfa)^{t_i}}{t_i!} \ll \left(\frac{\log{x}}{\log\log{x}}\right)^{D} (\log\log{x})^{-1},
\]
using the bound (\ref{trivialomegabound}) for the other choices of $i$. In all cases, our count of $(\pi)$ is at most
\[
\prod_{i = 1}^h \frac{u_i^{t_i}}{t_i!} \Big( \frac{\log x}{h \log\log x} \Big)^D + o\Big( \Big( \frac{\log x}{h \log\log x} \Big)^D \Big).
\]
Summing over the maximal types $\tau$, we conclude that
\begin{align}\label{Pupperbound}
\nu(\mfa) \le (P(u_1, \ldots, u_h) + o(1))\Big( \frac{\log x}{h \log\log x} \Big)^D.
\end{align}

Since
\[
\sum_{i = 1}^h u_i = \sum_{i:~u_i \leq 1/\log\log x} u_i + \sum_{i:~u_i > 1/\log\log x} u_i \leq h + o(1),
\]
the point $(u_1, \ldots, u_h)$ gets arbitrarily close to a point of $\Delta$ as $x\to\infty$. By the definition of $M$ and the uniform continuity of $P$ (on a closed set slightly larger than $\Delta$), we have that
\[ P(u_1, \ldots, u_h) \leq M + o(1).\] Inserting this into (\ref{Pupperbound}) completes the proof.

\begin{example}[analysis of the main term in Theorem \ref{thm:1stmain} when $\Cl(K)$ is cyclic] Suppose that $\Cl(K)$ is cyclic of order $h$. Number the ideal classes as $\cC_1, \dots, \cC_h$, where $\cC_i$ corresponds to $i\bmod{h}$ under a fixed isomorphism of   $\Cl(K)$ with $\Z/h\Z$. It is known that $D(\Cl(K))=h$ and that the maximal types are $(0,\dots,0,h,0,\dots,0)$, where the $h$ may appear in any of the $\phi(h)$ positions $1 \le i \le h$ with $\gcd(i,h)=1$ (see \cite[Corollary 2.1.4, p. 24]{GR09}). Thus,
\[ P(x_1,\dots,x_h) = \sum_{\substack{1 \le i \le h \\ \gcd(i,h)=1}} x_i^h/h!. \]
It is easily seen that the maximum of $P$ on $\Delta$ occurs when $x_1=h$ and all other $x_i$ vanish, so that
\[ M = h^h/h!. \]
We conclude that
\[ \max_{\alpha:~0<|N(\alpha)|\le x} \nu(\alpha) = (1+o(1)) \cdot \frac{1}{h!} \Big( \frac{\log x}{\log\log x} \Big)^h, \]
as $x\to\infty$.
\end{example}

\section{Irreducibles in arithmetic progressions: Proof of Theorem \ref{thm:2ndmain}}

Using the notation of Theorem \ref{thm:2ndmain}, let $\mfg := (\alpha, \mfm)$ be the gcd ideal of $(\alpha)$ and $\mfm$. Clearly, each $(\pi)$ counted in Theorem \ref{thm:2ndmain} is divisible by $\mfg$. To proceed with the proof of that theorem, we need to tailor the structure theory of irreducibles introduced in \S\ref{sec:anatomy} to irreducibles divisible by $\mfg$.

We call a type $\tau = (t_1,\dots, t_h)$ \emph{principal} if $\cC_1^{t_1} \cdots \cC_h^{t_h}$ is trivial in $\Cl(K)$. If $\tau$ and $\tau'$ are any two types, we say that \emph{$\tau'$ is a subtype of $\tau$}, and write $\tau' \preceq \tau$, if each component of $\tau'$ is less than or equal to the corresponding component of $\tau$. Thus, an irreducible type is a principal type with no nonzero principal subtype.

Observe that if a type $\tau'$ has no nonzero principal subtype, then $\tau' \preceq \tau$ for some (not necessarily unique) irreducible type $\tau$. Indeed, if $\tau'=(t_1',\dots,t_h')$ is not itself irreducible, we may take $\tau = (t_1', \ldots, t_{j-1}', t_j' + 1, t_{j+1}', \ldots, t_h')$, where $j$ is chosen so that $\cC_1^{t_1'} \cdots \cC_h^{t_h'} = \cC_j^{-1}$.  An irreducible type $\tau$ is said to be \emph{maximal with respect to} $\tau'$ if $\tau' \preceq \tau$ and the length of $\tau$ is maximal among those irreducible types having $\tau'$ as a subtype.


We now return to the situation of Theorem \ref{thm:2ndmain}. Since $(\alpha)$ and $\mfm$ are weakly relatively prime, the type $\tau'$ of $\mfg=(\alpha,\mfm)$ has no principal subtype. We now restate Theorem \ref{thm:2ndmain} in the form in which it will be proved, making explicit the constants in the asymptotic formula.

\begin{theorem}\label{irredprog} Let $\tau'$ be the type of $\mfg:=(\alpha,\mfm)$. As $x \to \infty$, we have
\[
\Pi(x; \mfm, \alpha) \sim \frac{1}{N(\mfg) \Phi(\mfm \mfg^{-1})} \frac{L}{h^{L}} \bigg(\sum_{\substack{ \tau' \preceq \tau \\ \tau~\text{max'l w.r.t.}~\tau'}} \frac{1}{t_1! \cdots t_h!}\bigg) \frac{x}{\log{x}} (\log \log x)^{L-1},
\]
where $\tau - \tau' = (t_1, \dots, t_h)$, and where $L$ is the (common) length of the types $\tau - \tau'$.

\end{theorem}

Theorem \ref{irredprog} follows immediately from the next proposition, upon summing over all types $\tau$ maximal relative to $\tau'$.

\begin{proposition}\label{irredspecifictype} Keep the notation of Theorem \ref{irredprog}. Let $\tau$ be a fixed irreducible type which is maximal with respect to $\tau'$. The number of ideals of norm at most $x$ generated by an irreducible element $\pi$ of type $\tau$ with $\pi \equiv \alpha \pmod \mfm$ and $\pi/\alpha \gg 0$ is asymptotically equal to
\begin{align}\label{irredspecifictypedisplay}
 \frac{1}{N(\mfg)\Phi(\mfm \mfg^{-1})} \frac{L}{h^L}\left(\prod_{j =1}^{h} \frac{1}{t_j!} \right)  \frac{x}{\log x} (\log\log x)^{L - 1},
\end{align}
where $\tau-\tau'=(t_1, \ldots, t_h)$ and $L=t_1+\dots+t_h$.
\end{proposition}

The proof of Proposition \ref{irredspecifictype} depends on the following lemma, which allows us to restrict our attention to ideals divisible by a large prime factor.

\begin{lemma}\label{bigprime}
Fix $k \in \bZ^+$. As $x \to \infty$, the number of (integral) ideals $\mfa$ with $N(\mfa) \le x$ having $k$ prime ideal factors (counted with multiplicity) which are not divisible by a prime ideal with norm exceeding $x^{1 - 1/\log\log x}$ is
\[
o\bigg( \frac{x (\log\log x)^{k - 1}}{\log x} \bigg).
\]
\end{lemma}

\begin{proof}[Proof of Lemma \ref{bigprime}] The number of ideals $\mfa$ with $N(\mfa) \leq x^{1 - \frac{1}{2\log\log x}}$ is $O(x^{1 - \frac{1}{2\log\log x}})$, which is negligible. Thus, we may restrict our attention to $\mfa$ with $N(\mfa) > x^{1 - \frac{1}{2\log\log x}}$. When $k=1$, there are no such $\mfa$ meeting the conditions of the lemma, and so we may assume that $k\ge 2$. Write $\mfa = \mfp_1 \cdots \mfp_k$, where $N(\mfp_1) \leq \cdots \leq N(\mfp_k)$. Then
\[
N(\mfp_k) > x^{1 - \frac{1}{2\log\log x}} N(\mfp_1 \cdots \mfp_{k-1})^{-1}.
\]
So if $N(\mfp_k) \leq x^{1 - 1/\log\log x}$, then
\[
\prod_{j = 1}^{k - 1} N(\mfp_j) \geq x^{\frac{1}{2\log\log x}}.
\]
This implies that $N(\mfp_{k-1}) \geq x^{\frac{1}{2k\log\log x}}$. We fix $\mfp_1, \ldots, \mfp_{k-1}$ and count the number of corresponding values of $\mfp_k$. Since $x/N(\mfp_1 \cdots \mfp_{k-1}) \geq N(\mfp_k) \geq x^{1/2k}$ (since $N(\mfp_k)^k \geq N(\mfa) \geq x^{1/2}$), we can use the prime ideal theorem to estimate the number of possibilities for $\mfp_k$ given $\mfp_1, \ldots, \mfp_{k-1}$ as
\[
\ll \frac{x}{\log x} \frac{1}{N(\mfp_1 \cdots \mfp_{k-1})}.
\]
We now sum on $\mfp_1, \ldots, \mfp_{k-1}$. By the prime ideal theorem (with the error term of Theorem \ref{thm:landau}) and partial summation, we have
\[
\sum_{N(\mfp) \leq x} \frac{1}{N(\mfp)} = \log\log x + O(1);
\]
this provides an upper bound for the sum on $\mfp_j$ for $1 \leq j \leq k - 2$. For the sum on $\mfp_{k-1}$, we use the estimate
\[
\sum_{x^{\frac{1}{2k\log\log x}} \leq N(\mfp) \leq x} \frac{1}{N(\mfp)} = \log\log\log x + O(1).
\]
Thus, we obtain an upper bound on the number of $\mfa$ that is
\[
\ll \frac{x(\log\log x)^{k-2}}{\log x} \log\log\log x,
\]
which implies the lemma.
\end{proof}

The following lemma will reduce the problem of counting our ideals $(\pi)$ to one of counting ideals of a certain type lying in a specific  strict ray class.

\begin{lemma}\label{rayclass} Let $\mfa$ be a nonzero integral ideal of $\Z_K$ divisible by $\mfg$. Then $\mfa = (\rho)$ for some $\rho \equiv \alpha \pmod \mfm$ with $\rho/\alpha \gg 0$ if and only if $\mfa \mfg^{-1}$ and $(\alpha) \mfg^{-1}$ represent the same element of $\Cl_{\mfm \mfg^{-1}}(K)$.
\end{lemma}

\begin{proof} First, suppose that $\mfa$ admits a generator $\rho \equiv \alpha \pmod \mfm$ with $\rho/\alpha \gg 0$. Then
\begin{equation}\label{eq:IGrel}
\mfa \mfg^{-1} = (\rho/\alpha) \cdot ((\alpha)\mfg^{-1}).
\end{equation}
We claim that
\begin{equation}\label{eq:rhoalpha}\rho/\alpha \equiv 1~\mathrm{mod}^{+}~\mfm \mfg^{-1}.\end{equation} By assumption, $\rho/\alpha \gg 0$, so it remains to show that $\ord_\mfp(\rho/\alpha - 1) \geq \ord_\mfp(\mfm \mfg^{-1})$ for all $\mfp \mid \mfm \mfg^{-1}$. If $\mfp$ divides $\mfm \mfg^{-1}$, then $\ord_\mfp(\mfg) < \ord_\mfp(\mfm)$; since $\ord_\mfp(\mfg) =\min\{\ord_\mfp(\alpha),\ord_{\mfp}(\mfm)\}$, it must be that $\ord_\mfp(\mfg) = \ord_\mfp(\alpha)$, and
\begin{align*}
\ord_\mfp(\mfm \mfg^{-1}) = \ord_\mfp(\mfm) - \ord_\mfp(\mfg) &= \ord_\mfp(\mfm) - \ord_\mfp(\alpha) \\
&\leq \ord_\mfp(\rho - \alpha) - \ord_\mfp(\alpha) = \ord_\mfp(\rho/\alpha - 1).
\end{align*}

In view of \eqref{eq:IGrel} and \eqref{eq:rhoalpha} , $\mfa \mfg^{-1}$ and $(\alpha)\mfg^{-1}$ represent the same element of $\Cl_{\mfm \mfg^{-1}}(K)$, as long as $\ord_{\mfp}(\mfa \mfg^{-1})=\ord_{\mfp}((\alpha)\mfg^{-1})=0$ for all  $\mfp \mid \mfm \mfg^{-1}$. That $\ord_{\mfp}((\alpha)\mfg^{-1})=0$ for all $\mfp \mid \mfm \mfg^{-1}$ is clear, since $\mfg$ is the gcd of $(\alpha)$ and $\mfm$. Since $\ord_{\mfp}(\rho/\alpha-1) > 0$ for all $\mfp \mid \mfm \mfg^{-1}$, the strong triangle inequality yields
\[ \ord_{\mfp}(\rho/\alpha) = \ord_{\mfp}((\rho/\alpha-1)+1) = 0 \]
for such $\mfp$. Thus, \eqref{eq:IGrel} implies that $\ord_{\mfp}(\mfa \mfg^{-1})=0$ for all $\mfp \mid \mfm \mfg^{-1}$.

We now turn to the converse implication. Suppose that $\mfa \mfg^{-1}$ and $(\alpha)\mfg^{-1}$ represent the same element of $\Cl_{\mfm \mfg^{-1}}(K)$. Then $\mfa \mfg^{-1} = \gamma (\alpha) \mfg^{-1}$, where $\gamma \in K$ satisfies $\gamma \equiv 1~\mathrm{mod}^{+}~{\mfm \mfg^{-1}}$. Thus $\mfa = \gamma \alpha \bZ_K$; since $\mfa$ is integral, $\rho:=\gamma\alpha \in \bZ_K$. The proof is completed by showing that $\rho/\alpha \gg 0$ and that $\rho \equiv \alpha\pmod{\mfm}$.

Since $\gamma \equiv 1~\mathrm{mod}^{+}~{\mfm \mfg^{-1}}$, we have $\rho/\alpha = \gamma \gg 0$. To prove the congruence for $\rho$ mod $\mfm$, we start by noticing if $\mfp \mid \mfm \mfg^{-1}$, then
\begin{align*}
\ord_\mfp(\rho - \alpha) = \ord_\mfp(\alpha) + \ord_\mfp(\gamma - 1) &\geq \ord_\mfp(\alpha) + \ord_\mfp(\mfm \mfg^{-1}) \\ &= \ord_\mfp(\alpha) + \ord_\mfp(\mfm) - \ord_\mfp(\mfg) \geq \ord_\mfp(\mfm).
\end{align*}
If $\mfp \mid \mfm$ but $\mfp \nmid \mfm \mfg^{-1}$, then $\ord_{\mfp}(\mfg) = \ord_{\mfp}(\mfm)$. Since $\ord_{\mfp}(\mfg) = \min\{\ord_{\mfp}(\alpha), \ord_{\mfp}(\mfm)\}$, it follows that $\ord_{\mfp}(\alpha) \ge \ord_{\mfp}(\mfm)$. Moreover, since $\mfg\mid \mfa = (\rho)$, we know that for these same $\mfp$,
\[ \ord_{\mfp}(\rho) \ge \ord_{\mfp}(\mfg) =\ord_{\mfp}(\mfm); \]
hence,
\[ \ord_{\mfp}(\rho-\alpha) \ge \min\{\ord_{\mfp}(\rho),\ord_{\mfp}(\alpha)\} \ge \ord_{\mfp}(\mfm). \]
Putting the above arguments together shows that $\ord_{\mfp}(\rho-\alpha)\ge \ord_{\mfp}(\mfm)$ for all $\mfp \mid \mfm$, and so $\rho \equiv \alpha\pmod{\mfm}$.
\end{proof}

\subsection{Proof of Proposition \ref{irredspecifictype}.} We are to count the number of ideals $(\pi)$ of bounded norm where $\pi$  is an irreducible of type $\tau$ with $\pi \equiv \alpha \pmod \mfm$ and $\pi/\alpha \gg 0$. Rather than count these $(\pi)$ directly, it is more convenient to count values of $\mfj:=(\pi)\mfg^{-1}$.

By Lemma \ref{rayclass}, the integral ideals $\mfj$ which arise are precisely those of type $\tau-\tau'$ with $\mfj$ relatively prime to $\mfm \mfg^{-1}$ and lying in the strict ray class of $(\alpha)\mfg^{-1}$ modulo $\mfm \mfg^{-1}$. Moreover, the condition that $N(\pi)\le x$ corresponds to the constraint that $N(\mfj) \le X:=x/N(\mfg)$.

We first prove an upper bound on the number of these $\mfj$ that matches the asymptotic formula claimed in Proposition \ref{irredspecifictype}.

Since $\mfj$ has type $\tau-\tau'$, which is of length $L$, the $\mfj$ under consideration are products of $L$ prime ideals (possibly with repetition). By Lemma \ref{bigprime}, we can assume one of these prime ideals has norm exceeding $X^{1 - 1/\log\log X}$, at the cost of excluding $o( X (\log\log X)^{L - 1} / \log X)$ values of $\mfj$, which is negligible. Thus,
\[
\mfj = \mfp_1 \cdots \mfp_L,
\]
where the prime ideals $\mfp_i$ are comaximal with $\mfm \mfg^{-1}$, and where $N(\mfp_L) > X^{1 - 1/\log\log X}$.

Recall that $t_1,\dots,t_h$ are defined by the equation $\tau-\tau'=(t_1,\dots,t_h)$. Since $\mfj$ has type $\tau-\tau'$, the class $\cC_i$ of $\mfp_L$  must satisfy $t_i>0$. We fix an index $i$ ($1\le i \le h$) with $t_i > 0$ and bound the number of $\mfj$ with $\mfp_L$ belonging to $\cC_i$.

Write $\mfj = \mfj_0 \mfp_L$. Then $\mfj_0$ is a product of $L-1$ prime ideals (with multiplicity), $t_j$ of which belong to the class $\cC_j$ for $j\ne i$, and $t_i-1$ of which belong to the class $\cC_i$. Moreover, given $\mfj_0$, the strict ray class of $\mfp_L$ modulo $\mfm \mfg^{-1}$ is uniquely determined as the class of $(\alpha) \mfg^{-1} \mfj_0^{-1}$. Since $N(\mfj_0 \mfp_L) = N(\mfj) \le X$, we also have that
\[
N(\mfp_L) \leq \frac{X}{N(\mfj_0)}.
\]
Noting that $X/N(\mfj_0) \ge N(\mfp_L) \ge X^{1-1/\log\log{X}}$, Landau's equidistribution theorem (for strict ray classes modulo $\mfm \mfg^{-1}$) implies that the number of possibilities for $\mfp_L$ given $\mfj_0$ is
\begin{align*}
\le (1+o(1)) \frac{X/N(\mfj_0)}{h_{K, \mfm \mfg^{-1}} \log(X/N(\mfj_0))} \leq \Big(\frac{1}{h_{K, \mfm \mfg^{-1}}} + o(1)\Big) \frac{1}{N(\mfj_0)} \frac{x}{N(\mfg) \log x}.
\end{align*}
(Recall that $X=x/N(\mfg)$.) Now sum on $\mfj_0$. The contribution of nonsquarefree values of $\mfj_0$ to $\sum\frac{1}{N(\mfj_0)}$ is zero unless $L \ge 3$, in which case it is
\[ \le \left(\sum_{N(\mfp)\le x} \frac{1}{N(\mfp)^2}\right)  \left(\sum_{N(\mfp)\le x} \frac{1}{N(\mfp)}\right)^{L-3}\ll (\log\log{x})^{L-3}. \]
The contribution of squarefree values of $\mfj_0$ to $\sum\frac{1}{N(\mfj_0)}$ is, by the multinomial theorem,
\[ \le \left(\prod_{\substack{1\le j \le h \\ j \ne i}} \frac{1}{t_j!} \Bigg(\sum_{\substack{N(\mfp) \le x \\ \mfp \in \cC_j}} \frac{1}{N(\mfp)}\Bigg)^{t_j}\right) \cdot \frac{1}{(t_i-1)!} \Bigg(\sum_{\substack{N(\mfp) \le x \\ \mfp \in \cC_i}} \frac{1}{N(\mfp)}\Bigg)^{t_i-1}.  \]
By Landau's theorem and partial summation,
$\sum_{\substack{N(\mfp) \leq x \\ \mfp \in \cC}} \frac{1}{N(\mfp)} = \frac{1}{h} \log\log x + O(1)$ for any ideal class $\cC$. Putting this in above, we find that
\[ \sum \frac{1}{N(\mfj_0)}\le (1+o(1)) \left(t_i \prod_{j=1}^{h}\frac{1}{t_j!}\right) \left(\frac{1}{h}\log\log{x}\right)^{L-1}, \]
so that the number of corresponding $\mfj$ is
\[ \le (1 + o(1)) \frac{1}{N(\mfg) h_{K,\mfm \mfg^{-1}} h^{L-1}} \left(t_i \prod_{j=1}^{h}\frac{1}{t_j!}\right) \frac{x}{\log x} (\log\log{x})^{L-1}. \]
Now sum on $i$ with $t_i\ne 0$. Since $\sum_{i:~t_i \ne 0} t_i = \sum_{i} t_i = L$, we conclude that the total number of $\mfj$ is
\[ \le (1 + o(1)) \frac{L}{N(\mfg) h_{K,\mfm \mfg^{-1}} h^{L-1}} \left(\prod_{j=1}^{h}\frac{1}{t_j!}\right) \frac{x}{\log x} (\log\log{x})^{L-1}. \]
Recalling that $\Phi(\mfm \mfg^{-1}) = h_{K,\mfm \mfg^{-1}}/h$, we see that \[ \frac{L}{N(\mfg) h_{K,\mfm \mfg^{-1}} h^{L-1}} = \frac{1}{N(\mfg) \Phi(\mfm \mfg^{-1})} \cdot \frac{1}{h^{L}}.\]
Making this substitution in the previous display,  our upper bound matches the expression \eqref{irredspecifictypedisplay} from Proposition \ref{irredspecifictype}.

The proof of the lower bound in Proposition \ref{irredspecifictype} is similar. Fix $i \in \{1, 2, \ldots, h\}$ with $t_i > 0$. Let $\mfp_1, \ldots, \mfp_{L-1}$ be distinct prime ideals not dividing $\mfm \mfg^{-1}$ whose norms belong to the interval $[2, X^{1/\log\log X}]$, with $t_j$ of these $\mfp_j$ belonging to the class $\cC_j$ for $j \neq i$, and $t_i - 1$ of the $\mfp_j$ belonging to the class $\cC_i$. Let $\mfp_L$ be a prime ideal with
\[
X^{1/2} \leq N(\mfp_L) \leq X/N(\mfp_1 \cdots \mfp_{L-1})
\]
belonging to the strict ray class modulo $\mfm \mfg^{-1}$ of $(\alpha) \mfg^{-1} \mfp_1^{-1} \cdots \mfp_{L - 1}^{-1}$. Define
\[ \mfj = \mfp_1 \cdots \mfp_L. \]

Let us see that $\mfj$ satisfies the conditions set down at the start of the proof. The ideal $\mfg \mfj$ is principal, since it is lies in the same strict ray class modulo $\mfm \mfg^{-1}$ as the principal ideal $(\alpha)$. This implies that $\mfp_L \in \cC_i$. Indeed, let $\mfp$ be a prime ideal from $\cC_i$. Then $\mfp_1 \cdots \mfp_{L-1} \mfp$ has type $(t_1,\dots,t_h)=\tau-\tau'$, so that $\mfg \mfp_1 \cdots \mfp_{L-1} \mfp$ has type $\tau$. Since $\tau$ is an irreducible type, $\mfg \mfp_1 \cdots \mfp_{L-1} \mfp$
is principal, and hence so is
\[ (\mfg \mfp_1 \cdots \mfp_{L-1} \mfp) (\mfg \mfj)^{-1} = \mfp \mfp_L^{-1}. \]
But this is only possible if $\mfp_L \in \cC_i$. It follows that $\mfj$ has type $\tau-\tau'$. By construction, $\mfj$ is relatively prime to $\mfm \mfg^{-1}$, lies in the strict ray class of $(\alpha)\mfg^{-1}$ modulo $\mfm \mfg^{-1}$, and satisfies $N(\mfj) \le X$.

So the proof of Proposition \ref{irredspecifictype} will be complete if we show that the number of $\mfj$ yielded by this construction is as large as the expression (\ref{irredspecifictypedisplay}). By Landau's equidistribution theorem, given $\mfp_1,\dots, \mfp_{L-1}$, the number of possibilities for $\mfp_L$ is at least
\[
(1 + o(1)) \frac{x}{h_{K, \mfm \mfg^{-1}} N(\mfg) \log x} \frac{1}{N(\mfp_1 \cdots \mfp_{L - 1})},
\]
as $x\to\infty$.
We again sum on possible tuples $\mfp_1, \ldots, \mfp_{L - 1}$. If we ignore the distinctness condition, then the sum on the $\mfp_j$ is at least
\[
\Big(\frac{1}{h}\log\log(X^{1/\log \log X}) + O(1)\Big)^{L-1} = (1+o(1)) \frac{1}{h^{L-1}} (\log \log x)^{L-1}.
\]
But the terms with $\mfp_j = \mfp_{j'}$ for any pair $1 \leq j \neq j' \leq L - 1$ contribute only $O((\log\log X)^{L - 3})$ to the sum. Hence, we may omit the distinctness condition without changing the asymptotic formula for the sum. We divide by $(t_i-1)! \prod_{1\le j \le h,~j\ne i}t_j!$ to avoid overcounting and conclude as in the proof of the upper bound.

\begin{example} Let $K = \Q(\sqrt{-23})$, $\alpha = \frac{1+\sqrt{-23}}{2}$ and $\mfm = (3)$. In this case, $\Cl(K) \cong \Z/3\Z$ (so that $h=3$), $(\alpha) = \mfp_1 \mfq_1$, and $(\mfm) = \mfp_1 \mfp_2$,
where $\mfp_1, \mfp_2, \mfq_1$ are distinct prime ideals of order $3$ in $\Cl(K)$  of respective norms $3, 3$, and $2$. Hence, $\mfg = (\alpha,\mfm) = \mfp_1$.
We choose the numbering of the ideal classes $\cC_1, \cC_2, \cC_3$ so that $\cC_1 = [\mfp_1]$. Then $\mfg$ has type $\tau' = (1,0,0)$. There is a unique irreducible type $\tau$ relative to $\tau'$, namely $\tau=(3,0,0)$. Thus, $\tau-\tau' = (2,0,0)$ and $L=2$. We have $N(\mfg) = N(\mfp_1) = 3$ and 
\[ \Phi(\mfm \mfg^{-1}) = \Phi(\mfp_2) = \frac{h_{K,\mfp_2}}{h_K}. \]
Using the well-known formula for the strict ray class number appearing, for example, as Proposition 2.1 on \cite[p. 50]{childress09}, we find that $h_{K,\mfp_2}/h_K = 1$. Plugging all of this into Theorem \ref{irredprog}, we conclude that as $x\to\infty$,
\[ \Pi(x;(3),\frac{1}{2}(1+\sqrt{-23})) \sim \frac{1}{27} \frac{x}{\log x} \log \log x. \]
\end{example}

\section*{Acknowledgements}
Work of the first-named author is supported by NSF award DMS-1402268.

\providecommand{\bysame}{\leavevmode\hbox to3em{\hrulefill}\thinspace}
\providecommand{\MR}{\relax\ifhmode\unskip\space\fi MR }
\providecommand{\MRhref}[2]{%
  \href{http://www.ams.org/mathscinet-getitem?mr=#1}{#2}
}
\providecommand{\href}[2]{#2}

\end{document}